\documentclass[12pt]{amsart}

\usepackage{hyperref}
\usepackage{amsmath}
\usepackage{amsthm}
\usepackage{amssymb}
\usepackage{amsfonts}
\usepackage{latexsym}
\usepackage{setspace}
\usepackage{enumitem}
\setlist[itemize]{leftmargin=*}
\setlist[enumerate]{leftmargin=*}

\makeatletter
\@namedef{subjclassname@2010}{%
  \textup{2010} Mathematics Subject Classification}
\makeatother

\newtheorem{theorem}{Theorem}[section]
\newtheorem{corollary}[theorem]{Corollary}
\newtheorem{lemma}[theorem]{Lemma}
\newtheorem{problem}[theorem]{Problem}
\theoremstyle{definition}
\newtheorem{remark}[theorem]{Remark}
\newtheorem{example}[theorem]{Example}

\numberwithin{equation}{section}

\newcommand{\N}{\mathbb{N}}

\newcommand{\R}{\mathbb{R}}

\newcommand{\Q}{\mathbb{Q}}
\newcommand{\Z}{\mathbb{Z}}
\renewcommand{\pmod}[1]{\,\,(\operatorname{mod} #1)}
\newcommand{\norm}[1]{\|#1\|}

\renewcommand{\geq}{\geqslant}
\renewcommand{\leq}{\leqslant}

\let\oldenumerate=\enumerate
	\def\enumerate{
	\oldenumerate
	\setlength{\itemsep}{5pt}
	}
\let\olditemize=\itemize
	\def\itemize{
	\olditemize
	\setlength{\itemsep}{5pt}
	}

\allowdisplaybreaks

\begin{document}


\baselineskip=17pt



\title[$p$-adic quotient sets]{$p$-adic quotient sets}

\author[S.R.~Garcia]{Stephan Ramon Garcia}
\address{Department of Mathematics\\Pomona College\\610 N. College Ave., Claremont, CA 91711} 
\email{stephan.garcia@pomona.edu}
\urladdr{http://pages.pomona.edu/~sg064747}
\thanks{First author partially supported by NSF grant DMS-1265973.}

\author[Y.X.~Hong]{Yu Xuan Hong}

\author[F.~Luca]{Florian Luca}
\address{School of Mathematics\\University of the Witwatersrand\\Private Bag 3, Wits 2050, Johannesburg, South Africa\\
Max Planck Institute for Mathematics, Vivatgasse 7, 53111 Bonn, Germany\\
Department of Mathematics, Faculty of Sciences, University of Ostrava, 30 dubna 22, 701 03
Ostrava 1, Czech Republic}
\email{Florian.Luca@wits.ac.za}

\author[E.~Pinsker]{Elena Pinsker}

\author[C.~Sanna]{Carlo Sanna}
\address{Department of Mathematics\\Universit\'a degli Studi di Torino\\Via Carlo Alberto, 10, 10123 Torino, Italy}
\email{carlo.sanna.dev@gmail.com}
	
\author[E.~Schechter]{Evan Schechter}

\author[A.~Starr]{Adam Starr}

\begin{abstract}
For $A \subseteq \mathbb{N}$, the question of when $R(A) = \{a/a' : a, a' \in A\}$ is dense in the positive real numbers $\mathbb{R}_+$ has been examined by many authors over the years.  In contrast, the $p$-adic setting is largely unexplored.  We investigate conditions under which $R(A)$ is dense in the $p$-adic numbers.  Techniques from elementary, algebraic, and analytic number theory are employed in this endeavor.  We also pose many open questions that should be of general interest.
\end{abstract}



\maketitle
\markleft{S.~R. GARCIA et. al.}

\section{Introduction}
For $A \subseteq \N = \{1,2,\ldots\}$ let
$R(A) = \{a/a' : a,a' \in A\}$
denote the corresponding \emph{ratio set} (or \emph{quotient set}).  
The question of when $R(A)$ is dense in the positive
real numbers $\R_+$ has been examined by many authors over the years
\cite{BT, QSDE, Hedman, Hobby, Nowicki, Misik, MR2505807, Pomerance, Micholson, Starni, 4QSG, Erdos, BST, Salat, SalatC, ST, STC}.  Analogues in the Gaussian integers \cite{QGP}
and, more generally, in imaginary quadratic number fields \cite{Sittinger} have been considered.

Since $R(A)$ is a subset of the rational numbers $\Q$, there are other important
metrics that can be considered.  
Fix a prime number $p$ and observe that each nonzero rational number has a unique representation of the form
$r = \pm p^k a/b$, in which $k \in \Z$, $a,b\in \N$ and $(a,p) = (b,p) = (a,b) =1$.  
The \emph{$p$-adic valuation} of such an $r$ is $\nu_p(r) = k$ and
its \emph{$p$-adic absolute value} is 
$\norm{r}_p = p^{-k}$.  By convention, $\nu_p(0) = \infty$ and $\norm{0}_p = 0$.  
The \emph{$p$-adic metric} on $\Q$ is $d(x,y) = \norm{x-y}_p$.
The field $\Q_p$ of \emph{$p$-adic numbers} is the completion 
of $\Q$ with respect to the $p$-adic metric.  
Further information can be found in \cite{Gouvea,Koblitz}.	

Garcia and Luca~\cite{QFN} recently proved that the set of quotients of Fibonacci numbers is dense in $\Q_p$ for all $p$.
Their result has been extended by Sanna~\cite{San17}, who proved that the quotient set of the $k$-generalized Fibonacci numbers is dense in $\Q_p$, for all integers $k \geq 2$ and primes $p$.
Other than these isolated results, the study of quotient sets in the $p$-adic setting 
appears largely neglected.
We seek here to initiate the general study
of $p$-adic quotient sets.
Techniques from elementary, algebraic, and analytic
number theory are employed in this endeavor.  
We also pose many open questions that should be of general interest.

Section \ref{Section:Preliminaries} introduces several simple, but effective, lemmas.
In Section \ref{Section:Real}, we compare and contrast the $p$-adic setting with the ``real setting.''
The potential $p$-adic analogues (or lack thereof) of known results from the real setting are discussed.

Sums of powers are studied in Section \ref{Section:Powers}.  
Theorem \ref{Theorem:SOS} completely describes for which $m$ and primes $p$
the ratio set of $\{x_1^2 + \cdots + x_m^2 : x_i \geq 0\}$ is dense in $\Q_p$.  Cubes are considerably 
trickier; the somewhat surprising answer is given by Theorem \ref{Theorem:Cubes}.

In Section \ref{Section:Recurrences} we consider sets whose elements are generated by a second-order recurrence.
Theorem \ref{Theorem:Recurrence} provides an essentially complete answer for recurrences of the form
$a_{n+2} = r a_{n+1} + s a_n$ with $a_0 = 0$ and $a_1 = 1$.  Examples are given that demonstrate the sharpness of our result.

Fibonacci and Lucas numbers are considered in Corollary \ref{Corollary:Fibonacci} which recovers the main result of \cite{QFN}:
the set of quotients of Fibonacci numbers is dense in each $\Q_p$.
The situation for Lucas numbers is strikingly different: the set of quotients of Lucas numbers is dense in $\Q_p$ if and only if $p \neq 2$ and
$p$ divides a Lucas number.

In Section \ref{Section:Progression} we examine certain unions of arithmetic progressions.
For instance, the ratio set of
$A = \{ 5^j : j \geq 0\} \cup \{7^j : j \geq 0\}$  is dense in $\Q_7$ but not $\Q_5$.  This sort of
asymmetry is not unusual.
Section \ref{Section:Proof} is devoted to the proof of Theorem \ref{Theorem:Sieve}, which asserts that
there are infinitely many pairs of primes $(p,q)$ such that $p$ is not a primitive root modulo 
$q$ while $q$ is a primitive root modulo $p^2$.  The proof is somewhat technical and
involves a sieve lemma due to Heath-Brown, along with
a little heavy machinery in the form of the Brun-Titchmarsh and Bombieri-Vinogradov theorems.
The upshot of Theorem \ref{Theorem:Sieve} is that
there are infinitely many pairs of primes $(p,q)$ so that the ratio set of
$\{ p^j : j \geq 0\} \cup \{q^k : k \geq 0\}$ is dense in $\Q_p$ but not in $\Q_q$.
A number of related questions are posed at the end of Section \ref{Section:Progression}.

\section{Preliminaries}\label{Section:Preliminaries}

We collect here a few preliminary observations and lemmas that will be fruitful in what follows.
Although these observations are all 
elementary, we state them here explicitly as lemmas since we will 
refer to them frequently.

\begin{lemma}\label{Lemma:Norm}
If $S$ is dense in $\Q_p$, then for each finite value of the $p$-adic valuation,
there is an element of $S$ with that valuation.
\end{lemma}

\begin{proof}
If $q \in \Q_p^{\times}$ can be
arbitrarily well approximated with elements of $S$, then there is a sequence $s_n \in S$
so that $|p^{-\nu_p(s_n)} - p^{-\nu_p(q)}| 
= | \norm{s_n}_p - \norm{q}_p | \leq \norm{s_n - q}_p \to 0$.
On $\Q^{\times}$, the $p$-adic valuation assumes only integral values, so
$\nu_p(s_n)$ eventually equals $\nu_p(q)$.
\end{proof}

The converse of the preceding lemma is false as 
$S = \{p^k : k \in \Z\}$ demonstrates.  More generally, we have the following lemma.

\begin{lemma}\label{Lemma:Geometric}
If $A$ is a geometric progression in $\Z$,
then $R(A)$ is not dense in any $\Q_p$.
\end{lemma}

\begin{proof}
If $A = \{ cr^n : n \geq 0\}$, in which $c$ and $r$ are nonzero integers, then
$R(A) = \{ r^n : n \in \Z\}$.
Let $p$ be a prime.
If $p \nmid r$, then $R(A)$ is not dense in $\Q_p$ by Lemma \ref{Lemma:Norm}.
If $p \mid r$, then $r^k \equiv -1 \pmod{p^2}$ is impossible 
since $-1$ is a unit modulo $p$.  Thus, $R(A)$ is bounded away from $-1$ in $\Q_p$.
\end{proof}

To simplify our arguments, we frequently appeal to the 
``transitivity of density.''  That is, if $X$ is dense in $Y$ and $Y$ is dense in $Z$, 
then $X$ is dense in $Z$.  This observation is used in conjunction with the following lemma.

\begin{lemma}\label{Lemma:Dense} 
Let $A \subseteq \N$.
\begin{enumerate}
\item If $A$ is $p$-adically dense in $\N$, then $R(A)$ is dense in $\Q_p$.
\item If $R(A)$ is $p$-adically dense in $\N$, then $R(A)$ is dense in $\Q_p$.
\end{enumerate}
\end{lemma}

\begin{proof}
\noindent(a) 
If $A$ is $p$-adically dense in $\N$, it is $p$-adically dense in $\Z$.
Inversion is continuous on $\Q_p^{\times}$, so $R(A)$ is $p$-adically dense in $\Q$,
which is dense in $\Q_p$.

\noindent(b) 
Suppose that $R(A)$ is $p$-adically dense in $\N$.
Since inversion is continuous on $\Q_p^{\times}$, the result
follows from the fact that $\N$ is $p$-adically dense in $\{x \in \Q : \nu_p(x) \geq 0\}$.
\end{proof}

Although the hypothesis of (a) implies the hypothesis of (b), we have stated them
separately.  We prefer to use (a) whenever possible.  We turn to (b) when confronted
with problems that do not succumb easily to (a).
This can occur since the hypothesis of (a) is not necessary for $R(A)$ to be dense in $\Q_p$.
If $A$ is the set of even numbers, then
$R(A) = \Q$ is dense in $\Q_p$ for all $p$,
but $A$ is not $2$-adically dense in $\N$.  The following lemma
concerns more general arithmetic progressions.

    \begin{lemma}\label{Lemma:AP}
        Let $A = \{ an+b : n \geq 0\}$. 
        \begin{enumerate}
            \item If $p\nmid a$, then $R(A)$ is dense in $\Q_p$.
            \item If $p \mid a$ and $p \nmid b$, then $R(A)$ is not dense in $\Q_p$.
        \end{enumerate}
    \end{lemma}
    
    \begin{proof}
    \noindent(a)
        Let $p \nmid a$ and let $n \in \N$ be arbitrary.  If $r \geq 1$, let
        $i \equiv a^{-1}(n - b) \pmod{p^r}$ so that
        $ai+b \equiv n \pmod{p^r}$.  Then $A$ is $p$-adically dense in $\N$, 
        so $R(A)$ is dense in $\Q_p$ by Lemma \ref{Lemma:Dense}.

        \noindent(b)
        If $p \mid a$ and $p \nmid b$, then
        $\nu_p(an+b) = 0$ for all $n$.
        Thus, $R(A)$ is not dense in $\Q_p$ by Lemma \ref{Lemma:Norm}.
    \end{proof}
    
\section{Real versus $p$-adic setting}\label{Section:Real}

A large amount of work has been dedicated to studying the behavior of 
quotient sets in the ``real setting.'' 
By this, we refer to work focused on determining
conditions upon $A$ which ensure that $R(A)$ is dense in $\R_+$.  It is
therefore appropriate to begin our investigations by examining the extent to which
known results in the real setting remain valid in the $p$-adic setting.

\subsection{Independence from the real case}
The behavior of a quotient set in the $p$-adic setting is essentially
independent from its behavior in the real setting.  To be more specific,
a concrete example exists
for each of the four statements of the form
``$R(A)$ is (dense/not dense) in every $\Q_p$ and 
    (dense/not dense) dense  in $\R_+$.''

    \begin{enumerate}
    \item Let $A = \N$.  Then $R(A)$ is dense in every $\Q_p$ and dense in $\R_+$.
    
    \item Let $F = \{1,2,3,5,8,13,21,34,55,\ldots\}$ denote the set of Fibonacci numbers.
    Then $R(F)$ is dense in each $\Q_p$ \cite{QFN};
    see Theorem \ref{Theorem:Recurrence} for a more general result.
      On the other hand, 
    Binet's formula ensures that $R(F)$ accumulates
    only at integral powers of the Golden ratio, so $R(F)$ is not dense in $\R_+$.
    
    \item Let $A = \{2,3,5,7,11,13,17,19,\ldots\}$ denote the set of prime numbers.
    The $p$-adic valuation of a quotient of primes belong to $\{-1,0,1\}$, 
    so Lemma \ref{Lemma:Norm} ensures that $R(A)$ is not dense in any $\Q_p$.
    The density of $R(A)$ in $\R_+$ is well known consequence of the Prime Number Theorem;
    see \cite[Ex.~218]{DM}, \cite[Ex.~4.19]{FR}, \cite[Cor.~4]{QSDE},
    	\cite[Thm.~4]{Hobby}, \cite[Ex.~7, p.~107]{Pollack}, \cite[Thm.~4]{Ribenboim}, \cite[Cor.~2]{Starni}
    	(this result dates back at least to Sierpi\'nski, who attributed it to Schinzel
    	\cite{Nowicki}).
    
    \item Let $A = \{2,6,30,210,\ldots\}$ denote the set of \emph{primorials}; 
    the $n$th primorial is the product of the first $n$ prime numbers.
    The $p$-adic valuation of a quotient of primorials belongs to $\{-1,0,1\}$,
    so Lemma \ref{Lemma:Norm} ensures that $R(A)$ is not dense in any $\Q_p$.  Moreover,
    $R(A) \cap [1,\infty) \subseteq \N$, so $R(A)$ is not dense in $\R_+$.
    \end{enumerate}

\subsection{Independence across primes}
Not only is the behavior of a quotient set in the $p$-adic setting unrelated
to its behavior in the real setting,
the density of a quotient set in one $p$-adic number system is completely independent, 
in a very strong sense, from its density in another.

\begin{theorem}
For each set $P$ of prime numbers, there is an $A \subseteq \N$
so that $R(A)$ is dense in $\Q_p$ if and only if $p \in P$.  
\end{theorem}

\begin{proof}
Let $P$ be a set of prime numbers, let $Q$ be the set of prime numbers
not in $P$, and let $A = \{ a \in \N : \nu_q(a) \leq 1\,\,\forall q \in Q\}$.
Lemma \ref{Lemma:Norm} ensures that $R(A)$ is not dense in 
$\Q_q$ for any $q \in Q$. 
Fix $p \in P$, let $\ell \geq 0$, and let $n = p^k m \in \N$ with $p \nmid m$.
Dirichlet's theorem on primes in arithmetic progressions provides a prime
of the form $r = p^{\ell} j + m$.  Then $p^k r \in A$ and
$p^k r \equiv p^k(p^{\ell}j+m) \equiv p^k m \equiv n \pmod{p^{\ell}}$.
Since $n$ was arbitrary, Lemma \ref{Lemma:Dense} ensures that $R(A)$ is dense in $\Q_p$.
\end{proof}

\subsection{Arithmetic progressions}

There exists a set $A \subseteq \N$ which contains arbitrarily long arithmetic progressions
and so that $R(A)$ is not dense in $\R_+$ \cite[Prop.~1]{4QSG}.  On the other hand, 
there exists a set $A$ that contains no arithmetic progressions of length three and 
so that $R(A)$ is dense in $\R_+$ \cite[Prop.~6]{4QSG}.
The same results hold, with different examples, in the $p$-adic setting.  

\begin{example}[Arbitrarily long arithmetic progressions]
The celebrated Green-Tao theorem asserts that the set of primes contains arbitrarily long
arithmetic progressions \cite{GreenTao}.  However, its ratio set is dense in no $\Q_p$;
see (c) in Section \ref{Section:Real}.
\end{example}

A set without long arithmetic progressions can have a quotient set that is dense in some 
$\Q_p$.  Consider the set $A = \{2^n : n \geq 0\} \cup \{3^n : n \geq 0\}$,
which contains no arithmetic progressions of length three \cite{4QSG}.
The upcoming Theorem \ref{Theorem:PrimitiveRoot} implies that $R(A)$ is dense in $\Q_3$.
One can confirm that $R(A)$ is bounded away from $5$ in $\Q_2$,
so $R(A)$ is dense in $\Q_p$ if and only if $p = 3$.  
However, we can do much better.

\begin{theorem}
There is a set $A \subseteq \N$ that contains no arithmetic progression of 
length three and which is dense in each $\Q_p$.
\end{theorem}

\begin{proof}
    Fix an enumeration $(q_n,r_n)$ of the set of all pairs
    $(q,r)$, in which $q$ is a prime power and $0 \leq r < q$;
    observe that each of the the pairs $(q,0),(q,1),\ldots,(q,q-1)$ appears 
    exactly once in this enumeration.
    Construct $A$, initially empty, according to the following procedure.
    Include the first natural number $a_1$ for which $a_1 \equiv r_1 \pmod{q_1}$,
    then include the first $a_2$ so that $a_2 > a_1$ and $a_2 \equiv r_2 \pmod{q_2}$.
    Next, select $a_3 > a_2$ so that $a_3 \equiv r_3 \pmod{q_3}$
    and so that $a_1,a_2,a_3$ is not an arithmetic progression of length three.
    Continue in this manner, so that a natural number $a_n > a_{n-1}$ is produced
    in the $n$th stage so that
    $a_n \equiv r_n \pmod{q_n}$ and
    $a_1,a_2,\ldots,a_n$ contains no arithmetic progression of length three.
    Since $A = \{a_n : n \geq 1\}$ contains a complete
    set of residues modulo each prime power,
    it is $p$-adically dense in $\N$ for each prime $p$.
    Thus, $R(A)$ is dense in each $\Q_p$ by Lemma \ref{Lemma:Dense}.
    By construction, $A$ contains no arithmetic progression of length three.
\end{proof}

\subsection{Asymptotic density}\label{Section:Density}
For $A \subseteq \N$, define $A(x) = A \cap [1,x]$.  Then
$|A(x)|$ denotes the number of elements in 
$A$ that are at most $x$.  
The \emph{lower asymptotic density} of $A$ is 
	\begin{equation*}
		\underline{d}(A) = \liminf_{n\to\infty} \frac{ | A(n)| }{n}
	\end{equation*}
and the \emph{upper asymptotic density} of $A$ is
\begin{equation*}
	\overline{d}(A) = \limsup_{n\to\infty} \frac{ | A(n)| }{n}.
\end{equation*}
If $\underline{d}(A) = \overline{d}(A)$, then their
common value is denoted $d(A)$ and called the \emph{asymptotic density}
(or \emph{natural density}) of $A$.  In this case, 
$d(A) = \lim_{n \to \infty} |A(n)|/n$.
Clearly $0 \leq \underline{d}(A) \leq \overline{d}(A) \leq 1$.

A striking result of Strauch and T\'oth is that
if $\underline{d}(A) \geq \frac{1}{2}$,
then $R(A)$ is dense in $\R_+$ \cite{ST}; see also \cite{4QSG} for a detailed exposition.
That is, $\frac{1}{2}$ is a critical threshold in the sense that any subset of $\N$ that
contains at least half of the natural numbers has a quotient set that is dense in $\R_+$.
On the other hand, the critical threshold in the $p$-adic setting is $1$.

\begin{theorem}\label{Theorem:Alpha}\hfill
\begin{enumerate}
\item If $\overline{d}(A) = 1$, then $R(A)$ is dense in each $\Q_p$.
\item For each $\alpha \in [0,1)$, there is an $A \subseteq \N$ so that
$R(A)$ is dense in no $\Q_p$ and $\underline{d}(A) \geq \alpha$.
\end{enumerate}
\end{theorem}

\begin{proof}
\noindent(a)
Suppose that $\overline{d}(A) = 1$.
If $A$ contained no representative from some congruence class modulo
a prime power $p^r$, then $\overline{d}(A) \leq 1 - 1/p^r < 1$, a contradiction.
Thus, $A$ contains a representative from every congruence class modulo 
each prime power $p^r$.  
Let $n,r \in \N$ and select $a,b \in A$ so that 
$a \equiv n \pmod{p^r}$ and $b \equiv 1 \pmod{p^r}$.  Then
$a\equiv bn \pmod{p^r}$ and hence
$\nu_p(a/b - n) = \nu_p(a - bn) \geq r$.  Thus,
$R(A)$ is dense in $\Q_p$ by Lemma \ref{Lemma:Dense}.
\medskip

\noindent(b)
Let $\alpha \in (0,1)$, 
let $p_n$ denote the $n$th prime number, and let $r_n$ be so large that
$2^n \leq (1-\alpha) p_n^{r_n}$ for $n\geq 1$.  If
\begin{equation*}
A = \{ a \in \N : \nu_{p_n}(a) \leq r_n\,\,\forall n\},
\end{equation*}
then $R(A)$ is dense in no $\Q_p$ by Lemma \ref{Lemma:Norm}.  Since $A$ omits
every multiple of $p_n^{r_n}$, 
\begin{equation*}
\underline{d}(A) 
\geq 1 - \sum_{n=1}^{\infty} \frac{1}{p_n^{r_n}}
\geq 1 - \sum_{n=1}^{\infty} \frac{(1-\alpha)}{2^n} = \alpha.\qedhere
\end{equation*}
\end{proof}

It is also known that if $\underline{d}(A) + \overline{d}(A) \geq 1$, 
then $R(A)$ is dense in $\R_+$ \cite[p.~71]{ST}.
The $p$-adic analogue is false:
for $\alpha \geq \frac{1}{2}$, the set $A$ from Theorem \ref{Theorem:Alpha}
satisfies $\underline{d}(A) + \overline{d}(A) \geq 1$ and is dense in no $\Q_p$.
On the other extreme, we have the following theorem.

\begin{theorem}
There is an $A \subseteq \N$ with $d(A) = 0$ for which 
$R(A)$ dense in every $\Q_p$.
\end{theorem}

\begin{proof}
Let $q_n$ denote the increasing sequence $2,3,4,5,7,8,9,11,13,16,17,\ldots$ 
of prime powers.  Construct $A$ according to the following procedure.
Add the first $q_1$ numbers to $A$ (that is, $1,2 \in A$) and 
skip the next $q_1!$ numbers (that is, $3,4, \notin A$).
Then add the next $q_2$ numbers to $A$ (that is, $5,6,7 \in A$)
and skip the next $q_2!$ numbers (that is $8,9,10,11,12,13 \notin A$).
The rapidly increasing sizes of the gaps between successive blocks of elements of $A$
ensures that $d(A) = 0$.  Since $A$ contains 
arbitrarily long blocks of consecutive integers, it
contains a complete set of residues modulo each $q_n$.
Thus, $A$ is $p$-adically dense in $\N$ for each prime $p$ and hence
$R(A)$ is dense in each $\Q_p$
by Lemma \ref{Lemma:Dense}.
\end{proof}

\subsection{Partitions of $\N$}
If $\N = A \sqcup B$, then at least one of $R(A)$ or $R(B)$ is dense in $\R_+$ \cite{BST};
this is sharp in the sense that there exists a partition $\N = A \sqcup B \sqcup C$ so that none of
$R(A)$, $R(B)$, and $R(C)$ is dense in $\R_+$.  See \cite{4QSG} for a detailed exposition of these results.  Things are different in the $p$-adic setting.

\begin{example}
Fix a prime $p$ and let
\begin{equation*}
A = \{ p^j n \in \N: \text{$j$ even, $(n,p)=1$}\}
\quad \text{and} \quad
B = \{ p^j n \in \N : \text{$j$ odd, $(n,p)=1$}\}.
\end{equation*}
Then $A \cap B = \varnothing$, but neither $R(A)$ nor $R(B)$ is dense in $\Q_p$ by 
Lemma \ref{Lemma:Norm}.  
\end{example}

\begin{problem}
Is there a partition $\N = A \sqcup B$ so that $A$ and $B$ are dense in no $\Q_p$?
\end{problem}

\section{Sums of powers}\label{Section:Powers}

The representation of integers as the sum of squares dates back to antiquity,
although this study only truly flowered with the work of Fermat, Lagrange, and Legendre.
Later authors studied more general quadratic forms and also representations of 
integers as sums of higher powers.  From this perspective, it is natural to consider the
following family of problems.  Let
\begin{equation*}
A = \{ a \in \N : a = x_1^n + x_2^n + \cdots + x_m^n, \,
    x_i \geq 0\}.
\end{equation*}
    For what $m$, $n$, and $p$ is $R(A)$ dense in $\Q_p$?
For squares and cubes, Theorems \ref{Theorem:SOS} and
\ref{Theorem:Cubes} provide complete answers.

\begin{theorem}\label{Theorem:SOS}
    Let $S_n = \{ a \in \N : \text{$a$ is the sum of $n$ squares, with $0$ permitted}\}$.  
    \begin{enumerate}
        \item $R(S_1)$ is not dense in any $\Q_p$.
        \item $R(S_2)$ is dense in $\Q_p$ if and only if $p \equiv 1 \pmod{4}$.
        \item $R(S_n)$ is dense in $\Q_p$ for all $p$ whenever $n \geq 3$.
    \end{enumerate}
\end{theorem}

\begin{proof}
\noindent(a)
Let $p$ be a prime.  Then $2 \mid \nu_p(s)$ for all $s \in S_1$ and hence
$R(S_1)$ is dense in no $\Q_p$ by Lemma \ref{Lemma:Norm}.
\medskip

\noindent(b) There are three cases: 
(b1) $p = 2$;
(b2) $p\equiv 1 \pmod{4}$;
(b3) $p \equiv 3 \pmod{4}$.
\medskip

\noindent(b1) 
Since $\nu_2(3) = 0$ any element $a/b \in R(S_2)$
that is sufficiently close in $\Q_2$ to $3$ must have $\nu_2(a) = \nu_2(b)$.  Without loss of generality,
we may assume that $a$ and $b$ are odd.  Then 
$a \equiv b \equiv 1 \pmod{4}$ since $a,b \in S_2$, so
$a \equiv 3b \pmod{4}$ is impossible.  Thus, $R(S_2)$ is bounded away from $3$ in $\Q_2$.
\medskip

\noindent(b2) Let $p \equiv 1 \pmod{4}$.
By Lemma \ref{Lemma:Dense}, it suffices to show that 
for each $n \geq 0$ and $r \geq 1$, the congruence
$x^2 + y^2 \equiv n \pmod{p^r}$ has a solution with $p \nmid x$.
We proceed by induction on $r$.
Since there are precisely $(p+1)/2$ quadratic residues modulo $p$, the sets
$\{x^2 : x \in \Z/p\Z\}$ and $\{n - y^2 : y \in \Z/p\Z\}$ have a nonempty intersection.
Thus, $x^2 + y^2 \equiv n \pmod{p}$ has a solution.
If $p \nmid n$, then $p$ cannot divide both $x$ and $y$;
in this case we may assume that $p \nmid x$.  If $p \mid n$, let $x = 1$
and $y^2 \equiv -1 \pmod{p}$; such a $y$ exists since $p \equiv 1 \pmod{4}$.
This establishes the base case $r=1$.

Suppose that $x^2 + y^2 \equiv n \pmod{p^r}$ and $p \nmid x$.
Then $x^2+y^2 = n+ m p^r$ for some $m \in \Z$.  Let
$i \equiv -2^{-1}x^{-1} m \pmod{p}$ so that $p \mid (2ix+m)$.  Then
\begin{align*}
(x+ip^r)^2+y^2 
&= x^2 + 2ix p^r + i^2 p^{2r} + y^2 \\
&\equiv n + (2ix+ m)p^r \pmod{p^{r+1}} \\
&\equiv n \pmod{p^{r+1}}.
\end{align*}
This completes the induction.
\medskip

\noindent(b3)
Let $p \equiv 3 \pmod{4}$.  
If $a,b \in S_2$, then a theorem of Fermat ensures that 
$\nu_p(a)$ and $\nu_p(b)$ are both even.
Then $\nu_p(a) - \nu_p(b) = \nu_p(a/b) \neq 1 = \nu_p(p)$ for all $a,b \in S_2$.
Thus, $S_2$ is bounded away from $p$ in $\Q_p$.
\medskip

\noindent(c)
Lagrange's four-square theorem asserts that $S_n = \N$ for $n \geq 4$, so
$R(S_n) = \Q$ is dense in $\Q_p$ for $n \geq 4$.  Thus, we consider only $n=3$.
There are three cases to consider: 
(c1) $p = 2$;
(c2) $p\equiv 1 \pmod{4}$;
(c3) $p \equiv 3 \pmod{4}$.
\medskip

\noindent(c1) Recall that Legendre's three square theorem asserts that a natural number is in $S_3$
if and only if it is not of the form $4^i(8j+7)$ for some $i,j\geq 0$.  Consequently, if the $2$-adic order
of a natural number is odd, then it is the sum of three squares.
Let $n \in \N$ be odd and let $k \in \N$.
If $k$ is odd, let $a = 2^k n$ and $b = 1$;
if $k$ is even, let $a = 2^{k+1}n$ and $b = 2$.
Then $a = 2^k n b$ and $a,b \in S_3$ since $\nu_2(a)$ is odd.
Consequently, $a \equiv 2^k n b \pmod{2^r}$ for all $r \in \N$, so 
$R(S_3)$ is $2$-adically dense in $\N$.
Lemma \ref{Lemma:Dense} ensures that $R(S_3)$ is dense in $\Q_2$.
\medskip

\noindent(c2) If $p \equiv 1 \pmod{4}$, then $R(S_3)$ contains $R(S_2)$,
which is dense in $\Q_p$ by (b2).
\medskip

\noindent(c3) Let $p \equiv 3 \pmod{4}$.
Since $4^j(8k+7)$ is congruent to either $0$, $4$, or $7$ modulo $8$, it follows that
$S_3$ contains the infinite arithmetic progression $A = \{ 8k+1 : k \geq 0\}$.
Lemma \ref{Lemma:AP} ensures that $R(A)$ is dense in $\Q_p$, so 
$R(S_3)$ is dense in $\Q_p$ too.
\end{proof}

\begin{theorem}\label{Theorem:Cubes}
Let $C_n = \{a \in \N: \text{$a$ is the sum of $n$ cubes, with $0$ permitted}\}$.
\begin{enumerate}
\item $R(C_1)$ is not dense in any $\Q_p$.
\item $R(C_2)$ is dense in $\Q_p$ if and only if $p \neq 3$.
\item $R(C_n)$ is dense in all $\Q_p$ for all $p$ whenever $n \geq 3$.
\end{enumerate}
\end{theorem}

\begin{proof}
\noindent(a) Let $p$ be a prime.  Then $3 \mid \nu_p(c)$ 
for all $c \in C_1$, so Lemma \ref{Lemma:Norm} ensures that $R(C_1)$ is dense in no $\Q_p$.
\medskip

\noindent(b) There are three cases:  (b1) $p\neq 3,7$; (b2) $p=3$; and (b3) $p=7$.
\medskip

\noindent(b1) The congruence $x^3 + y^3 \equiv n \pmod{m}$ has a solution for each
$n$ if and only if $7 \nmid m$ or $9 \nmid m$ \cite[Thm.~3.3]{Broughan}.
Consequently, $C_2$ is $p$-adically dense in $\N$ if $p\neq 3,7$,
so $R(C_2)$ is dense in $\Q_p$ for $p \neq 3,7$ by Lemma \ref{Lemma:Dense}. 
\medskip

\noindent(b2) 
If $x/y \in R(C_2)$ is sufficiently close to $3$ in $\Q_3$, then
$\nu_3(x) = \nu_3(y) + 1$. Without loss of generality, we may suppose that
$\nu_3(x) = 1$ and $\nu_3(y) = 0$.  
A sum of two cubes modulo $9$ must be among
$0, 1, 2, 7, 8$, so $\nu_3(x) = 1$ is impossible for $x \in C_2$.
Thus, $R(C_2)$ is not dense in $\Q_3$.
\medskip

\noindent(b3) 
Let $p=7$.  For each integer $m$ congruent to $0,1,2,5$ or $6$ modulo $7$
and each $r \geq 1$,
we use induction on $r$ to show that $x^3 + y^3 \equiv m \pmod{7^r}$ has a solution
with $7 \nmid x$.  The cubes modulo $7$ are $0,1$ and $6$ and hence
each of the residue classes $0,1,2,5,6$ is a sum of two cubes
modulo $7$, at least one of which is nonzero.
This is the base of the induction.
Suppose that $n$ congruent to $0,1,2,5$ or $6$ modulo $7$ and that
$x^3 + y^3  \equiv m \pmod{7^r}$, in which $7 \nmid x$.
Then $x^3 + y^3  = m+ 7^r \ell$ for some $\ell \in \Z$.
Let $i \equiv -5\ell x^{-2} \pmod{7}$, so that $7 \mid (3ix^2 + \ell)$.
Then
\begin{align*}
(x + 7^r i )^3 + y^3  
&\equiv  x^3 + y^3  + 3x^2 \cdot 7^r i   \pmod{7^{r+1}} \\
&\equiv m + 7^r(3ix^2 + \ell)   \pmod{7^{r+1}} \\
& \equiv m \pmod{7^{r+1}}.
\end{align*}
Since $7\nmid x$, we have $7 \nmid(x+7^r i)$. This completes the induction.

The inverses of $1,2,5$ and $6$ modulo $7$ are $1,4,3$ and $6$, respectively.
Consequently, for each $m$ congruent to $1,3,4$ or $6$ modulo $7$,
the congruence $(x^3 + y^3)^{-1}\equiv m \pmod{7^r}$ has a solution with
$7 \nmid x$.  

Each residue class modulo $7$ is a product of an element in 
$\{0,1,2,5,6\}$ with an element in $\{1,3,4,6\}$.  Given a natural number $n$ and $r\geq 0$,
write $n \equiv m_1 m_2 \pmod{7^r}$, in which 
$m_1$ modulo $7$ is in $\{0,1,2,5,6\}$
and $m_2$ modulo $7$ is in  $\{1,3,4,6\}$.  Then there are $c_1,c_2 \in C_2$
so that $c_1c_2^{-1} \equiv m_1 m_2 \equiv n \pmod{7^r}$. 
Lemma \ref{Lemma:Dense} ensures that $R(C_2)$ is dense in $\Q_7$.
\medskip

\noindent(c) 
There are two cases to consider:
(c1) $n \geq 4$;
(c2) $n = 3$ and $p =3$.
\medskip

\noindent(c1) 
Almost every natural number, in the sense of natural density,
is the sum of four cubes \cite{Davenport}.  
For each prime power $p^r$ and each $n \in \N$,
the congruence $x \equiv n \pmod{p^r}$ must have a solution with $x \in C_4$ since
otherwise the natural density of $C_4$ would be at most $1 - 1/p^r$.
Lemma \ref{Lemma:Dense} ensures that $R(C_4)$ is dense in $\Q_p$.
\medskip

\noindent(c2) 
Modulo $9$, the set of cubes is $\{0,1,8\}$; modulo $9$,
the set of sums of three cubes is $\{0,1,2,3,6,7,8\}$.  
Since $4 \cdot 7 \equiv 5 \cdot 2 \equiv 1 \pmod{9}$,
each element of $\{1,4,5,8\}$ is the inverse, modulo $9$,
of a residue class that is the sum of three cubes.

A lifting argument similar to that used in the proof of (b3)
confirms that whenever $m \equiv 0,1,2,3,6,7,8 \pmod{9}$,
the congruence $x^3 + y^3 + z^3 \equiv m \pmod{3^r}$
has a solution with $3 \nmid x$ for all $r \geq 2$.  Thus,
whenever $m \equiv 1,4,5,8 \pmod{9}$, the congruence
$(x^3 + y^3 + z^3)^{-1} \equiv m \pmod{3^r}$ has a solution
with $3 \nmid x$ for all $r \geq 2$.

Each residue class modulo $9$ is the product of an element in
$\{0,1,2,3,6,7,8\}$ and an element of $\{1,4,5,8\}$.  Given
a natural number $n$ and $r \geq 2$, write
$n \equiv m_1 m_2 \pmod{3^r}$, in which 
$m_1$ modulo $9$ is in $\{0,1,2,3,6,7,8\}$
and $m_2$ modulo $9$ is in $\{1,4,5,8\}$.  Then there are
$c_1,c_2 \in C_3$ so that $c_1 c_2^{-1} \equiv m_1 m_2 \equiv n \pmod{3^r}$.
Lemma \ref{Lemma:Dense} ensures that $R(C_2)$ is dense in $\Q_3$.
\end{proof}

We close this section with a couple open problems.
In light of Theorems \ref{Theorem:SOS} and
\ref{Theorem:Cubes}, the following question is the next logical step.

\begin{problem}
What about sums of fourth powers?  Fifth powers?
\end{problem}

Turning in a different direction, instead of sums of squares one might
consider quadratic forms.  Cubic and biquadratic forms might also eventually be considered.

\begin{problem}
    Let $Q$ be a quadratic form and let $A = \{ a \in \N: \text{$Q$ represents $a$}\}$.
    For which primes $p$ is $R(A)$ dense in $\Q_p$?
\end{problem}

\section{Second-order recurrences}\label{Section:Recurrences}

Garcia and Luca \cite{QFN} showed that the set of quotients of 
Fibonacci numbers is dense in $\Q_p$ for all $p$.  The proof employed
a small amount of algebraic number theory and some relatively obscure
results about Fibonacci numbers.  The primes $p = 2$ and $p = 5$
required separate treatment.  
This result has been extended by Sanna~\cite{San17}, who proved that the quotient set of the $k$-generalized Fibonacci numbers is dense in $\Q_p$, for all integers $k \geq 2$ and all primes $p$.
The proof made use of $p$-adic analysis.
In this section, we establish a result for
certain second-order recurrences that include the Fibonacci numbers as a special case.

Given two fixed integers $r$ and $s$, let $(a_n)_{n \geq 0}$ be defined by
\begin{equation*}
a_0 = 0, \qquad a_1 = 1, \qquad a_{n+2} = r a_{n+1} + s a_n,
\end{equation*}
and let $(b_n)_{n \geq 0}$ be defined by
\begin{equation*}
b_0 = 2, \qquad b_1 = r, \qquad b_{n+2} = r b_{n+1} + s b_n .
\end{equation*}
Sequences like $(a_n)_{n \geq 0}$ and $(b_n)_{n \geq 0}$ are usually known as \emph{Lucas sequences of the first kind} and \emph{second kind}, respectively.
We assume that the \emph{characteristic polynomial} $x^2 - rx - s$ has two nonzero roots $\alpha, \beta \in \mathbb{C}$ such that $\alpha / \beta$ is not a root of unity.
In particular, $\alpha$ and $\beta$ are distinct.
Under this hypothesis, it is well known that
\begin{equation*}
a_n = \frac{\alpha^n - \beta^n}{\alpha - \beta} \quad \text{and} \quad b_n = \alpha^n + \beta^n ,
\end{equation*}
for all integers $n \geq 0$, and that $a_n, b_n \neq 0$ for all $n \geq 1$.
For each prime number $p$ such that $p \nmid s$, let $\tau(p)$ be the \emph{rank of appearance} of $p$ in $(a_n)_{n \geq 0}$, that is, the smallest positive integer $k$ such that $p \mid a_k$ (such $k$ exists, see \cite{MR3141741}).
Regarding the $p$-adic valuation of $(a_n)_{n \geq 0}$, we have the following result~\cite[Theorem~1.5]{MR3512829}.

\begin{theorem}\label{Theorem:padican}
If $p \nmid s$, then
\begin{equation*}
\nu_p(a_n) = \begin{cases}
\nu_p(n) + \nu_p(a_p) - 1 & \text{ if } p \mid \Delta,\; p \mid n, \\
0 & \text{ if } p \mid \Delta,\; p \nmid n, \\
\nu_p(n) + \nu_p(a_{p \tau(p)}) - 1 & \text{ if } p \nmid \Delta,\; \tau(p) \mid n,\; p \mid n, \\
\nu_p(a_{\tau(p)}) & \text{ if } p \nmid \Delta,\; \tau(p) \mid n,\; p \nmid n, \\
0 & \text{ if } p \nmid \Delta,\; \tau(p) \nmid n ,
\end{cases}
\end{equation*}
for all positive integers $n$, where $\Delta = r^2 + 4s$.
\end{theorem}

Actually, in~\cite{MR3512829} Theorem~\ref{Theorem:padican} is proved under the hypothesis that $s$ and $r$ are relatively prime, but it can be readily checked that the proof works exactly in the same way also if $s$ and $r$ are not coprime.

Our result is the following:

\begin{theorem}\label{Theorem:Recurrence}
Let $A_n = \{a_n : n \geq 1\}$ and $B_n = \{b_n : n \geq 1\}$.
\begin{enumerate}
\item If $p \mid s$ and $p \nmid r$, then $R(A)$ is not dense in $\Q_p$.\footnote{If $p \mid s$ and $p \mid r$, then anything can happen.  See the remarks after the proof.}
\item If $p\nmid s$, then $R(A)$ is dense in $\Q_p$.
\item For all odd primes $p$, we have that $R(B)$ is dense in $\Q_p$ if and only if there exists a positive integer $n$ such that $p \mid b_n$.
\end{enumerate}
\end{theorem}
\begin{proof}
\noindent(a) If $p \mid s$ and $p \nmid r$, then induction confirms that $a_n \equiv r^{n-1} \pmod{p}$
so that $\nu_p(a_n) = 0$ for $n\geq 0$.  Thus, $R(A)$ is not
dense in $\Q_p$ by Lemma \ref{Lemma:Norm}.  

\medskip
\noindent(b)
Suppose that $p \nmid s$.
From Theorem~\ref{Theorem:padican} it follows quickly that for any integer $j \geq 1$ there exists an integer $m \geq 1$ such that $\nu_p(a_m) \geq j$.
Therefore,
\begin{equation*}
a_m = \frac{\alpha^m - \beta^m}{\alpha - \beta} \equiv 0 \pmod{p^j} ,
\end{equation*}
so that
\begin{equation}\label{eq:alphambetam}
\alpha^m \equiv \beta^m \pmod{p^j} ,
\end{equation}
where we write $x \equiv y \pmod t$ to mean that $(x - y) / t$ is an algebraic integer.
Setting $k = 2mp^{j-1}(p-1)$, by (\ref{eq:alphambetam}) we get
\begin{equation*}
\alpha^k \equiv \beta^k \equiv (\alpha^m \beta^m)^{p^{j-1}(p-1)} \equiv (-s)^{mp^{j-1}(p-1)} \equiv 1 \pmod{p^j} ,
\end{equation*}
since $\alpha\beta = -s$, $p \nmid s$, and thanks to the Euler's theorem.
Hence,
\begin{align*}
\frac{a_{kn}}{a_k} &= \frac{(\alpha^k)^n - (\beta^k)^n}{\alpha^k - \beta^k} \\
&= (\alpha^k)^{n-1} + (\alpha^k)^{n-2}(\beta^k) + \cdots + (\beta^k)^{n-1} \\
&\equiv n \pmod{p^j} ,
\end{align*}
for all positive integers $n$.
Since $n$ and $j$ are arbitrary, Lemma \ref{Lemma:Dense} ensures that $R(A)$ is dense in $\Q_p$.

\medskip\noindent(c)$(\Rightarrow)$ We prove the contrapositive. 
If $p\nmid b_n$ for all $n \geq 1$, then $\nu_p(b_n) = 0$
for all $n \geq 1$.  
Then $R(B)$ is not dense in $\Q_p$ by Lemma \ref{Lemma:Norm}.

\medskip\noindent(c)$(\Leftarrow)$ Suppose that $p$ is an odd prime that divides $b_n$ for some $n \geq 1$.
Since $b_n = a_{2n} / a_n$, we have
\begin{equation*}
\nu_p(b_n) = \nu_p(a_{2n}) - \nu_p(a_n).
\end{equation*}
Hence, $\nu_p(a_{2n}) > \nu_p(a_{n})$.
Keeping in mind that $p$ is odd, from Theorem~\ref{Theorem:padican} it follows that
\begin{equation}\label{eq:lotdivs}
p \nmid \Delta, \quad \tau(p) \nmid n, \quad \tau(p)  \mid  2n ,
\end{equation}
which in turn implies that
\begin{equation}\label{eq:taunu2}
\nu_2(\tau(p)) = \nu_2(n) + 1 .
\end{equation}
For an integer $j \geq 1$, put
\begin{equation*}
k = p^{j+1} \cdot \frac{p - 1}{2^{\nu_2(p - 1)}} \cdot n .
\end{equation*}
On the one hand, by (\ref{eq:taunu2}) we have $\tau(p) \nmid k$, so that Theorem~\ref{Theorem:padican} yields $\nu_p(a_k) = 0$.
On the other hand, $\tau(p) \mid 2k$ and $p \mid 2k$. Hence, by Theorem~\ref{Theorem:padican},
\begin{equation*}
\nu_p(a_{2k}) = \nu_p(2k) + \nu_p(a_{p\tau(p)}) - 1 \geq j .
\end{equation*}
Therefore,
\begin{equation*}
\nu_p(b_k) = \nu_p(a_{2k}) - \nu_p(a_k) \geq j ,
\end{equation*}
so that
\begin{equation*}
\alpha^k \equiv -\beta^k \pmod{p^j} .
\end{equation*}
Setting $h = 2^{\nu_2(p - 1)}$, we have
\begin{align*}
(\alpha^k)^{2h} &\equiv (\alpha^k)^h (\alpha^k)^h \equiv (\alpha^k)^h (-\beta^k)^h  \\
&\equiv (\alpha^k)^h (\beta^k)^h \equiv (-s)^{p^{j+1}(p-1)n} \equiv 1 \pmod{p^j} ,
\end{align*}
since $\alpha\beta=-s$, $p \nmid s$, and thanks to Euler's theorem.
Noting that $2h$ and $p^j$ are relatively prime, by the Chinese remainder theorem, for any positive integer $m$, we can pick a positive integer $\ell$ such that
\begin{equation*}
\ell \equiv m \pmod{p^j} \quad\text{and}\quad \ell \equiv 1 \pmod{2h} .
\end{equation*}
Therefore, since $\ell$ is odd,
\begin{align*}
\frac{b_{k\ell}}{b_k} &= \frac{(\alpha^{k})^\ell + (\beta^{k})^\ell}{(\alpha^{k}) + (\beta^{k})} \\
&= (\alpha^{k})^{\ell-1} + (\alpha^{k})^{\ell-2}(-\beta^{k}) + \cdots + (-\beta^{k})^{\ell-1} \\
&\equiv \ell (\alpha^{k})^{\ell - 1} \equiv m \pmod{p^j} .
\end{align*}
Since $j,n$ are arbitrary, Lemma \ref{Lemma:Dense} tell us that $R(B)$ is dense in $\Q_p$.
\end{proof}

Theorem~\ref{Theorem:Recurrence} has the following corollary, whose first part is the main result of~\cite{QFN}.

\begin{corollary}\label{Corollary:Fibonacci}
Let $(F_n)_{n \geq 0}$ be the sequence of Fibonacci numbers and let $(L_n)_{n \geq 0}$ be the sequence of Lucas numbers.
Put also $F = \{F_n : n \geq 1\}$ and $L = \{L_n : n \geq 1\}$.
\begin{enumerate}
\item $R(F)$ is dense in $\Q_p$, for all $p$.
\item $R(L)$ is dense in $\Q_p$ if and only if $p \neq 2$ and $p \mid L_n$ for some $n \geq 1$.
\end{enumerate}
\end{corollary}
\begin{proof}
Pick $r = s = 1$.

\medskip\noindent(a) 
Follow immediately from Theorem~\ref{Theorem:Recurrence}(b).

\medskip\noindent(b)
For odd $p$, the claim is a consequence of Theorem~\ref{Theorem:Recurrence}(c).
For $p = 2$, it is enough to note that $(L_n)_{n \geq 0}$ is periodic modulo $8$ and $8 \nmid L_n$ for all $n \geq 1$. Hence, the claim follows from Lemma~\ref{Lemma:Norm}.
\end{proof}

The set of primes that divide some Lucas number $L_n$ has relative
density $2/3$ as a subset of the primes \cite{Lagarias}; see \cite{Ballot} for more information
about these types of results.  In particular, $R(L)$ is dense in $\Q_p$
more often than not.

Theorem~\ref{Theorem:Recurrence}(a--b) is sharp in the following sense:
if $p \mid s$ and $p \mid r$, then $R(A)$ may or may not be dense in $\Q_p$.
Consider the following examples.

\begin{example}
Let $p=3$, $r=15$, and $s = -54$; note that $p \mid s$ and $p \mid r$. 
Then $\alpha = 9$ and $\beta = 6$, so
\begin{equation*}
a_n=\frac{9^n-6^n}{9-6} = 3^{n-1}(3^n-2^n).
\end{equation*}
We claim that $R(A)$ is not dense in $\Q_3$.   
Since $\nu_3(3) = 1$ and $\nu_3( a_m / a_n) = (m-1)-(n-1) = m-n$ for $m,n \geq 0$, 
any element of $R(A)$ that is sufficiently close to $3$ in $\Q_3$
must be of the form $a_{n+1} / a_n$ for some $n \geq 1$.  However,
\begin{align*}
\nu_3\Big( \frac{a_{n+1}}{a_n} - 3 \Big)
&= \nu_3\Big(  \frac{3^n(3^{n+1} - 2^{n+1})}{3^{n-1}(3^n - 2^n)} - 3\Big) \\
&=1 +  \nu_3\Big( \frac{3^{n+1} - 2^{n+1}}{3^n - 2^n} -1\Big) \\
&= 1 + \nu_3( 3^{n+1} - 2^{n+1} - 3^n + 2^n) \\
&= 1 + \nu_3\Big( 3^n(3-1) - 2^n(2-1) \Big) \\
&= 1 + \nu_3\Big( 2 \cdot 3^n - 2^n \Big) 
= 1 + \nu_3\Big( 3^n - 2^{n-1} \Big) = 1,
\end{align*}
so $R(A)$ is bounded away from $3$ in $\Q_3$.
\end{example}

\begin{example}
Let $p = 5$, $r = 20$, and $s = -75$; note that $p \mid s$ and $p \mid r$.
Then $\alpha = 15$ and $\beta = 5$, so
\begin{equation*}
a_n = \frac{15^n - 5^n}{15-5} = 5^{n-1} \frac{3^n-1}{2}.
\end{equation*}
We claim that $R(A)$ is dense in $\Q_5$.
Let $N \in \N$ be given and write $N = 5^t N_0$, in which $5 \nmid N_0$.
Let $r \geq t$ be large and so that $r\not\equiv t-1\pmod 4$.
Since $\phi(5^{r+1}) = 4 \cdot 5^r$,
Euler's theorem permits us to write
\begin{equation}\label{eq:345k}
3^{4\cdot 5^r}-1=5^{r+1} \ell, \qquad 5\nmid \ell.
\end{equation}
Let $m \geq 1$ satisfy
\begin{equation}\label{eq:mlrN}
m \equiv \ell^{-1} N_0(3^{r-t+1}-1) \pmod{5^{r+1}};
\end{equation}
Euler's theorem ensures that $5 \nmid m$ since $4 \nmid (r-t+1)$.
If $n=4\cdot 5^r m$, then 
\begin{align*}
&(3^{n+r-t+1}-1)\frac{a_n}{a_{n+r-t+1}}  \\
&\qquad= 5^{n-1} \Big(\frac{3^n-1}{2}\Big)  \Big(\frac{2}{5^{n+r-t}} \Big)\\
&\qquad =  5^{-r+t-1}( 3^{4\cdot 5^r m}-1 )\\
&\qquad =  5^{t}\Big(\frac{3^{4\cdot 5^r m}-1}{3^{4\cdot 5^r}-1}\Big) 
\Big(\frac{3^{4\cdot 5^r}-1}{5^{r+1}}\Big)\\
&\qquad =  5^{t} \Big(\frac{(3^{4\cdot 5^r})^m-1}{3^{4\cdot 5^r}-1}\Big) \ell && (\text{by \eqref{eq:345k}})\\
&\qquad=5^{t} \big( (3^{4\cdot 5^r})^{m-1} + (3^{4\cdot 5^r})^{m-2}+\cdots + 1\big) \ell\\
&\qquad\equiv 5^t m \ell \pmod{5^{r+1}}  &&(\text{since $\phi(5^{r+1}) = 4 \cdot 5^r$})\\
&\qquad\equiv 5^t N_0(3^{r-t+1}-1) \pmod{5^{r+1}} && (\text{by \eqref{eq:mlrN}})\\
&\qquad\equiv N(3^{r-t+1}-1) \pmod{5^{r+1}} && (\text{since $N = 5^t N_0$})\\
&\qquad\equiv N(3^{n+r-t+1}-1) \pmod{5^{r+1}} && (\text{since $\phi(5^{r+1})|n$}).
\end{align*}
Since $4 \mid n$ and $4\nmid (r-t+1)$, it follows that
$5\nmid(3^{n+r-t+1}-1)$ and hence
\begin{equation*}
\nu_5\Big( \frac{a_n}{a_{n+r-t+1}} - N \Big) \geq r+1.
\end{equation*}
Thus, $R(A)$ is $5$-adically dense in $\N$, so it is dense in $\Q_5$
by Lemma \ref{Lemma:Dense}.
\end{example}

Where do the preceding two examples leave us?  
Suppose that $p \mid s$ and $p \mid r$.
Then $d = (r,s)$ is divisible by $p$.  Induction confirms that
$d^{\lfloor n/2 \rfloor}$ divides $a_n$ for $n \geq 0$.
To be more specific, 
\begin{equation*}
a_n = \sum_{k=0}^{n-1} \binom{n-1-k}{k} r^{ \max\{n-1-2k,0\}} s^{\max\{ \lfloor \frac{2k+1}{2} \rfloor, 0\}}
\end{equation*}
and hence it is difficult to precisely evaluate $\nu_p(a_n)$.  Consequently,
we are unable at this time to completely characterize when $R(A)$ is 
dense in $\Q_p$ if $p \mid s$ and $p \mid r$.  However, in most instances, ad hoc arguments
can handle these situations.  Consider the following examples.

\begin{example}
The integer sequence
\begin{equation*}
0,\, 1,\, 2,\, -1,\, -12,\, -19,\, 22,\, 139,\, 168,\, -359,\, -1558,\, -1321,\, 5148,\ldots
\end{equation*}
is generated by the recurrence 
\begin{equation*}
a_0 = 0,\qquad a_1 = 1, \qquad
a_{n+2} = 2a_{n+1} -5a_n, \quad n \geq 0.
\end{equation*}
In this case, $\alpha = 1+2i$ and $\beta = 1- 2i$ are nonreal.
Let $A = \{a_n : n \geq 1\}$ and confirm by induction that
$5 \nmid a_n$ for all $n \geq 1$.  Apply Lemma \ref{Lemma:Norm} and
Theorem \ref{Theorem:Recurrence} to see that $R(A)$ is dense in $\Q_p$
if and only if $p \neq 5$.
\end{example}

\begin{example}
Let $b$ be an integer not equal to $\pm 1$ and consider
the sequence $a_n = b^n-1$, which is generated by 
\begin{equation*}
a_0 = 0,\qquad a_1 = b-1, \qquad
a_{n+2} = (b+1)a_{n+1} -ba_n.
\end{equation*}
Apply Theorem \ref{Theorem:Recurrence} with $r = b+1$ and $s = -b$ to the set
$A = \{a_n : n \geq 1\}$ and conclude
that $R(A)$ is dense in $\Q_p$ if and only if $p \nmid b$.  For instance, the
set $\{1,3,7,15,31,\ldots\}$ is dense in $\Q_p$ if and only if $p \neq 2$.
\end{example}

\section{Unions of geometric progressions}\label{Section:Progression}

The ratio set of $A = \{ 2^n : n \in \N\} \cup \{3^n : n \in \N\}$ is dense in $\R_+$ 
\cite[Prop.~6]{4QSG}.  The argument relies upon the irrationality of $\log_2 3$
and the inhomogeneous form of Kronecker's approximation theorem \cite[Thm.~440]{Hardy}.  
We consider such sets in the $p$-adic setting, restricting our attention to prime bases.
The reader should have no difficulty
stating the appropriate generalizations if they are desired.

    An integer $g$ is called a \emph{primitive root modulo $m$} if
    $g$ is a generator of the multiplicative group $(\Z/m\Z)^{\times}$.
    Gauss proved that primitive roots exist only for the moduli
    $2$, $4$, $p^k$, and $2p^k$, in which $p$ is an odd prime
    \cite[Thm.~2.41]{Niven}.  Consequently, our arguments here tend to focus on
    odd primes.

\begin{theorem}\label{Theorem:PrimitiveRoot}
Let $p$ be an odd prime, let $b$ be a nonzero integer, and let
$$A = \{ p^j : j \geq 0\} \cup \{b^j : j \geq 0\}.$$
Then $R(A)$ is dense in $\Q_p$ if and only if $b$ is a primitive root modulo $p^2$.
\end{theorem}

\begin{proof}
    \noindent($\Rightarrow$)
    Suppose that $R(A)$ is dense in $\Q_p$.
    We first claim that $b$ is a primitive root modulo $p$.
    If not, then there is an $m \in \{2,3,\ldots,p-1\}$
    so that $b^j \not\equiv m \pmod{p}$ for all $j \in \Z$.  
    Then $R(A)$ is bounded away from $m$ in $\Q_p$, a
    contradiction.  Thus, $b$ must be a primitive root modulo $p$.
    
    Suppose toward a contradiction that $b$ is not a primitive root modulo $p^2$.
    Since $b$ is a primitive root modulo $p$, the order of $b$ modulo $p^2$
    is at least $p-1$.  On the other hand, the order must divide $\Phi(p^2) = p(p-1)$.
    Since $p$ is prime, it follows that 
    the order of $b$ modulo $p^2$ must be exactly $p-1$.  Thus, $b^{p-1} \equiv 1 \pmod{p^2}$.
    
    If $b^n \equiv p+1 \pmod{p^2}$, then $b^n \equiv 1 \pmod{p}$ and hence
    $n$ is a multiple of $p-1$.  Then $b^n \equiv 1 \pmod{p^2}$, a contradiction.
    Thus, $R(A)$ is bounded away from $p+1$ in $\Q_p$.
    This contradiction shows that $b$ must be a primitive root modulo $p^2$.
    \medskip

    \noindent($\Leftarrow$) 
    Let $r\geq 1$ and let $n = p^k m \in \N$, in which $p \nmid m$.
    Since $b$ is a primitive root modulo $p^2$
    it is a primitive root modulo $p^3,p^4,\ldots$ \cite[Thm.~2.40]{Niven},
    so there is a $j$ such that $b^j m \equiv 1 \pmod{p^r}$.  Thus,
        \begin{equation*}
            \nu_p \Big( n - \frac{p^k}{b^j} \Big)
            = k + \nu_p \Big( m - \frac{1}{b^j} \Big)
            \geq \nu_p(b^j m - 1) \geq r,
        \end{equation*}
        so $R(A)$ is dense in $\Q_p$ by
        Lemma \ref{Lemma:Dense}.
\end{proof}

    A primitive root modulo $p$ is not necessarily a primitive root modulo $p^2$.
    For instance, $1$ is a primitive root modulo $2$, but not modulo $4$.  A less trivial example is
    furnished by $p=29$, for which $14$ is a primitive root modulo $p$, but not modulo $p^2$.
    Similarly, if $p=37$, then $18$ is a primitive root modulo $p$, but not modulo $p^2$.

\begin{example}\label{Example:ExactlyOne}
Consider $p = 5$ and $q=7$; note that  
$5$ is a primitive root modulo $7$ and vice versa.  
However, $5$ is a primitive root modulo $7^2$ but $7$ is not a primitive root modulo $5^2$.
Let $$A = \{5, 7, 25, 49, 125, 343, 625,2401, 3125,\ldots\}
= \{ 5^j : j \geq 0\} \cup \{7^j : j \geq 0\}.$$  Then
Theorem \ref{Theorem:PrimitiveRoot} ensures that $R(A)$ is dense in $\Q_7$
but not in $\Q_5$.
\end{example}

This sort of asymmetry is not unusual.  The following theorem tells us that
infinitely many such pairs of primes exist.
The proof is considerably more difficult than
the preceding material and it requires a different collection of tools
(e.g., a sieve lemma of Heath-Brown, the Brun-Titchmarsh theorem, and the
Bombieri-Vinogradov theorem).
Consequently,  the proof of Theorem \ref{Theorem:Sieve} is deferred until Section \ref{Section:Proof}.

\begin{theorem}\label{Theorem:Sieve}
There exist infinitely many pairs of primes $(p,q)$ such that $p$ is not a primitive root modulo $q$ and $q$ is a primitive root modulo $p^2$.
\end{theorem}

\begin{corollary}
    There are infinitely many pairs of primes $(p,q)$ so that the ratio set of
    $\{ p^j : j \geq 0\} \cup \{q^k : k \geq 0\}$ is dense in $\Q_p$ but not in $\Q_q$.
\end{corollary}

Let $a \prec b$ denote ``$p$ is a primitive root modulo $b$.''  The following table
shows some of the various logical possibilities (bear in mind that a primitive root modulo 
$p^2$ is automatically
a primitive root modulo $p$).
\begin{equation*}
\begin{array}{cc|cccc}
p & q & p\prec q& q\prec p & p \prec q^2 & q \prec p^2\\
\hline
3 & 5 & T & T & T & T \\
5 & 7 & T & T & T & F \\
3 & 7 & T & F & T & F\\
5 & 11 & F & F & F & F \\
7 & 19 & F & T & F & F \\
\end{array}
\end{equation*}
Despite an extensive computer search, we were unable to find a pair $(p,q)$ of primes 
for which $p \prec q$ and $q \prec p$, but $p \not\prec q^2$ and $q \not\prec p^2$.
We hope to revisit this question in later work.  For now we are content to pose the
following question.

\begin{problem}
Is there a pair $(p,q)$ of primes 
for which $p \prec q$ and $q \prec p$, but $p \not\prec q^2$ and $q \not\prec p^2$
\end{problem}

On the other hand, numerical evidence and heuristic arguments
suggests that there are infinitely many pairs $(p,q)$
of primes for which $p \prec q^2$ and $q \prec p^2$.  Although we have
given the proof of the closely related Theorem \ref{Theorem:Sieve}, we feel that attempting to address
more questions of this nature would draw us too far afield.  Consequently,
we postpone this venture until another day.  

\begin{problem}
Prove that there infinitely many pairs of primes $(p,q)$ for which $p \prec q^2$ and $q \prec p^2$?
\end{problem}

\section{Proof of Theorem \ref{Theorem:Sieve}}\label{Section:Proof}

\noindent\textbf{Step 1}:
We start with a sieve lemma due to Heath-Brown. 
It is a weaker version of \cite[Lem.~3]{HB}. In what follows, $p$
always denotes an odd prime.  We write $(\frac{a}{p})$ 
for the Legendre symbol of $a$ with respect to $p$ and write $x \prec y$
to indicate that $x$ is a primitive root modulo $y$.

\begin{lemma}[Heath-Brown]\label{Lemma:GHB}
Let $q,r,s$ be three primes, let $k\in \{1,2,3\}$, and let 
$u,v$ be positive integers such that
\begin{enumerate}
\item $16 \mid  v$,
\item $2^k \mid (u-1)$,
\item $(\frac{u-1}{2^k},v)=1$,
\item if $p\equiv u\pmod v$, then 
\begin{equation*}
\Big(\frac{-3}{p}\Big)=\Big(\frac{q}{p}\Big)=\Big(\frac{r}{p}\Big)=\Big(\frac{s}{p}\Big)=-1.
\end{equation*}
\end{enumerate}
Then for large $x$, the set of primes
\begin{align*}
\mathcal{P} (x;u,v) 
&= \big\{p\leq x \,:\, \text{$p\equiv u\pmod v$, $(p-1)/2^k$ is prime or a product of}  \\
&\qquad\quad  \text{two primes, and one of $q,r,s$ is primitive root modulo $p$}\big\}
\end{align*}
has cardinality satisfying
\begin{equation*}
|\mathcal{P} (x;u,v)|\gg \frac{x}{(\log x)^2}.
\end{equation*}
\end{lemma}

Consider the primes $q=7$, $r=11$, and $s=19$.
Let $$v=70{,}224=16\times 3\times 7\times 11\times 19$$
and observe that $16 \mid v$, so (a) is satisfied.  Let $u=2{,}951$ so that
$$u-1=2{,}950=2\times 5^2\times 59,$$ so that
(b) is satisfied with $k=1$.  Then (c) is satisfied since $u-1$ and $v$ have no
common factors.  
If $p\equiv u\pmod v$, then 
\begin{equation}\label{eq:BoCs}
p\equiv 3\pmod 4,\qquad
p\equiv 2\pmod 3,\quad  \text{and}\quad 
p\equiv 25\pmod {qrs}. 
\end{equation}
Since $q \equiv r \equiv s \equiv 3\pmod{4}$,
quadratic reciprocity ensures that $q,r,s$ are quadratic nonresidues
modulo $p$.  For instance,
\begin{equation}\label{eq:qrpq}
\Big(\frac{q}{p}\Big)=-\Big(\frac{p}{q}\Big)=-\Big(\frac{25}{p}\Big)=-1
\end{equation}
and similarity if $q$ is replaced with $r$ or $s$.
In addition,
\begin{equation*}
\Big(\frac{-3}{p}\Big)
=\Big(\frac{-1}{p}\Big) \Big(\frac{3}{p}\Big)
=\Big(\frac{p}{3}\Big)
=\Big(\frac{2}{3}\Big)=-1
\end{equation*}
by \eqref{eq:BoCs} and quadratic reciprocity.  Thus, (d) is satisfied.
\medskip

\noindent\textbf{Step 2}:
With $u=2{,}951$, $v = 70{,}224$,  $q=7$, $r=11$, and $s=19$ as above, let 
\begin{equation*}
\mathcal{P} =\{p\equiv u\pmod v\,:\,
\text{\rm one of $q,r,s$ is a primitive root modulo $p$}\}
\end{equation*}
and let $\mathcal{P} (x)=\mathcal{P} \cap [1,x]$.
Since $\mathcal{P} (x;u,v) \subseteq \mathcal{P}(x)$,
Lemma \ref{Lemma:GHB} ensures that
\begin{equation*}
|\mathcal{P} (x)|\geq \frac{x}{(\log x)^2}.
\end{equation*}
If $p \in \mathcal{P}$, then $(\frac{p}{q}) = 1$ by \eqref{eq:qrpq},
so $p$ is a quadratic residue modulo $q$
and hence it fails to be a primitive root modulo $q$.
Since \eqref{eq:qrpq} holds with 
$q$ replaced with $r$ or $s$, we conclude that
$p$ is not a primitive root modulo $q$, $r$, or $s$.
\medskip

\noindent\textbf{Step 3}:
The definition of $\mathcal{P}(x)$ ensures that
one of the primes $q,r,s$ is a primitive root 
for at least $\frac{1}{3}|\mathcal{P} (x)|$ primes $p\leq x$. 
Without loss of generality, we may assume that this prime is $q$
since the specific numerical values of $q,r,s$ are irrelevant in what follows
(save that they are all congruent to $3$ modulo $4$).
Let
\begin{equation*}
\mathcal{P} _1(x)
=
\Big\{\frac{x}{(\log x)^2}<p\leq x \,:\, 
\text{$p\equiv u\pmod v$ and $q \prec p$}\Big\}
\end{equation*}
and let $\pi(x)$ denote the number of primes at most $x$.
Then the prime number theorem implies that
\begin{align*}
|\mathcal{P} _1(x) |
& \geq  \tfrac{1}{3} |\mathcal{P} (x)|-\pi\left( \frac{x}{(\log x)^2}\right)\\
& \gg  \frac{x}{(\log x)^2}-O\left(\frac{x}{(\log x)^3}\right)\\
& \gg   \frac{x}{(\log x)^2}.
\end{align*}

\noindent\textbf{Step 4}:
Let $\mathcal{P} _{2}(x)$ be the set of $p \in \mathcal{P} _1(x)$ for which $q \prec p^2$;
let $\mathcal{P} _3(x)=\mathcal{P} _1(x)\backslash \mathcal{P} _2(x)$
be the subset for which $q \not\prec p^2$.  One of these two possibilities
must occur at least half the time.
This leads to two cases.
\medskip

\noindent\textbf{Step 4.a}:
If 
\begin{equation}\label{eq:2}
|\mathcal{P} _{2}(x)| \geq \tfrac{1}{2} |\mathcal{P} _{1}(x)|,
\end{equation}
then there are at least 
$\gg  x / (\log x)^2$ pairs of primes $(p,q)$ with $p \leq x$
for which $q \prec p^2$ and $p \not \prec q$.
\medskip

\noindent\textbf{Step 4.b}:
Suppose that 
\begin{equation}\label{eq:3}
    |\mathcal{P} _3(x)|\geq \tfrac{1}{2}|\mathcal{P} _1(x)| .
\end{equation}
For $p\in \mathcal{P} _3(x)$ consider primes of the form
\begin{equation}\label{eq:ell}
    \ell=q+4hp,\qquad  h\in [1,x^{3/2}-1]\cap \Z,
\end{equation}
first observing that $\ell \ll x^{5/2}$.
The prime number theorem for arithmetic progressions asserts that
the number of such primes is $\pi(4x^{3/2}p;4p,q)+O(1)$,
in which $\pi(x;m,a)$ denotes the number of primes at most $x$
that are congruent to $a \pmod{m}$.
For such a prime $\ell$, quadratic reciprocity and \eqref{eq:qrpq} ensure that
\begin{equation*}
    \Big(\frac{p}{\ell}\Big)
    =-\Big(\frac{\ell}{p}\Big)
    =-\Big(\frac{q+4hp}{p}\Big)
    =-\Big(\frac{q}{p}\Big)=1
\end{equation*}
since $p,q \equiv 3 \pmod{4}$.  Consequently, $p \not \prec \ell$.
Since
\begin{align*}
\ell^{p-1} 
& =  (q+4h p)^{p-1} \\
&\equiv q^{p-1}+4h(p-1) q^{p-2} p \pmod {p^2}\\
& \equiv  1+4h(p-1)q^{p-2} p\pmod {p^2},
\end{align*}
we see that $\ell$ is a primitive root modulo $p^2$ whenever $p \nmid h$.
The Brun-Titchmarsh theorem ensures that 
the number of primes $\ell$ of the form \eqref{eq:ell} for which $p \mid h$ is
\begin{equation*}
\pi(4x^{3/2}p;4p^2,q)+O(1)\ll \frac{x^{3/2} p}{p^2 \log(x^{3/2} p)}\ll \frac{x^{3/2}}{p\log x}.
\end{equation*}
Thus, for each $p \in \mathcal{P}_3(x)$, there are
\begin{equation*}
\pi(4x^{3/2}p;4p,q)-\pi(4x^{3/2}p;4p^2,q)+O(1)
\end{equation*}
primes $\ell \leq x^{5/2}$ for which $p \not \prec \ell$ and $\ell \prec p^2$.
\medskip

\noindent\textbf{Step 5}:
Let $\mathcal{P} _4(x)$ be the subset of $\mathcal{P} _3(x)$ such that 
\begin{equation}\label{eq:5}
\pi(4x^{3/2}p;4p,q)-\pi(4x^{3/2}p;4p^2,q)\geq \frac{x^{3/2}}{(\log x)^2},
\end{equation}
and let $\mathcal{P} _5(x)=\mathcal{P}_3(x)\backslash \mathcal{P} _4(x)$.
We intend to show that $|\mathcal{P} _5(x)|$ is small
and hence that \eqref{eq:5} holds for most $p \in \mathcal{P}_3(x)$.
Suppose that $p\in \mathcal{P} _5(x)$; that is,
\begin{equation*}
\pi(4x^{3/2}p;4p,q)-\pi(4x^{3/2}p;4p^2,q)< \frac{x^{3/2}}{(\log x)^2}.
\end{equation*}
Then 
\begin{align*}
&\left|\pi(4x^{3/2}p;4p;q) - \frac{\pi(4x^{3/2}p)}{\phi(4p)}\right| \\
&\qquad \geq  \frac{\pi(4x^{3/2}p)}{\phi(4p)} -  \big|\pi(4x^{3/2}p;4p,q) -  \pi(4x^{3/2}p;4p^2,q)\big|
 -  \pi(4x^{3/2}p;4p^2,q)\\
&\qquad \gg  \frac{x^{3/2}}{\log x}+O\left(\frac{x^{3/2}}{(\log x)^2}+\frac{x^{3/2}}{p\log x}\right)\\
&\qquad \gg  \frac{x^{3/2}}{\log x}
\end{align*}
because $p\geq x/(\log x)^2$ by the definition of $\mathcal{P}_1(x)$. 
Thus, 
\begin{equation*}
    \max_{\substack{1\leq a\leq 4p \\ (a,4p)=1\\ 1\leq y\leq 4x^{5/2}}}
     \Big|\pi(y;4p,a)-\frac{\pi(y)}{\phi(4p)}\Big|\geq 
    \Big|\pi(4x^{3/2} p;4p,q)-\frac{\pi(4x^{3/2}p)}{\phi(4p)}\Big|
    \gg \frac{x^{3/2}}{\log x}.
\end{equation*}
Summing up the above inequality over all $p\in \mathcal{P} _5(x)$
and appealing to the Bombieri-Vinogradov theorem, for large $x$ we obtain
\begin{align*}
\left(\frac{x^{3/2}}{\log x}\right) |\mathcal{P} _5(x)|
& \ll \sum_{p\in \mathcal{P} _5(x)} \,\,
 \max_{\substack{1\leq a\leq 4p\\ (a,4p)=1\\ 1\leq y\leq 4x^{5/2}}} 
 \Big|\pi(y;4p,a)-\frac{\pi(y)}{\phi(4p)}\Big| \\
& \ll \sum_{m\leq 4x} \max_{\substack{1\leq a\leq m:\\ (a,m)=1\\  y\leq x^{5/2}}} 
    \Big|\pi(y;m,a)-\frac{\pi(y)}{\phi(m)}\Big| \\
&\ll x^{5/2}/(\log x)^4.
\end{align*}
Thus,
\begin{equation*}
|\mathcal{P} _5(x)| \ll \frac{x}{(\log x)^3}
\end{equation*}
and hence
\begin{equation*}
 |\mathcal{P} _4(x)|
 =|\mathcal{P} _3(x)|-|\mathcal{P} _5(x)|
 \geq \tfrac{1}{2}|\mathcal{P} _3(x)|
 \gg \frac{x}{(\log x)^2}
\end{equation*}
when \eqref{eq:3} holds and $x$ is sufficiently large.
Comparing this with \eqref{eq:5}, we conclude that the number of pairs $(p,\ell)$ with 
$p \not \prec \ell$ and $\ell \prec p^2$ and $p<\ell\leq 4x^{5/2}$ is
\begin{equation}\label{eq:CSTTS}
 \gg \left(\frac{x^{3/2}}{(\log x)^2}\right)\left(\frac{x}{(\log x)^2}\right)\gg \frac{x^{5/2}}{(\log x)^4}.
\end{equation}

\noindent\textbf{Step 6}.
If \eqref{eq:2} holds, then the number of pairs $(p,q)$ for which $p \in \mathcal{P}_1(x)$,
$q \prec p^2$, and $p \not \prec q$ is $\gg x / (\log x)^2$, which is dominated
by \eqref{eq:CSTTS} for large $x$.  For such $p$ we have 
$\max\{p,q\}\leq x<4x^{5/2}$. If $y=4x^{5/2}$ is sufficiently large, then
the number of pairs of primes $(p,q)$ with $\max\{p,q\}\leq y$ 
and for which $p \not \prec q$ and $q \prec p^2$ is
\begin{equation*}
 \gg \frac{x}{(\log x)^2} \gg \frac{y^{2/5}}{(\log y)^2}. 
\end{equation*}
This completes the proof of Theorem \ref{Theorem:Sieve}. \qed
 
 \begin{remark}
A more careful application of the Bombieri-Vinogradov theorem shows that this count can be improved to $y^{1/2-\varepsilon}$ for any $\varepsilon>0$ fixed (just replace the range for $h$ in \eqref{eq:ell} $x^{3/2}$ by $x^{1+\varepsilon}$), or even to $y^{1/2}/(\log y)^A$ for some constant $A>0$, but we do not get into such details. \end{remark}

\section*{Acknowledgements}
F. L. was supported in part by grants CPRR160325161141 and an A-rated researcher award both from the NRF of South Africa and by grant no. 17-02804S of the Czech Granting Agency.

\bibliographystyle{amsplain}
\providecommand{\bysame}{\leavevmode\hbox to3em{\hrulefill}\thinspace}
\providecommand{\MR}{\relax\ifhmode\unskip\space\fi MR }
\providecommand{\MRhref}[2]{%
  \href{http://www.ams.org/mathscinet-getitem?mr=#1}{#2}
}
\providecommand{\href}[2]{#2}

\end{document}